\documentclass[11pt]{article}
\usepackage{amsmath}
\usepackage{amsthm}
\input{epsf.tex}
\newtheorem{theorem}{Theorem}[section]
\newtheorem{proposition}[theorem]{Proposition}

\newtheorem{define}[theorem]{Definition}

\def\Empty{}

\catcode`\@=11

\def\section{\@startsection {section}{1}{\z@}{-3.5ex plus -1ex minus 
-.2ex}{2.3ex plus .2ex}{\large\bf}}

\def\fnum@figure{{\small Figure \thefigure}}
\def\fakefigure{\def\@captype{figure}}

\long\def\@makecaption#1#2{
    \vskip 10pt 
    \def\FCap{#2} \def\NoCap{\ignorespaces}
    \ifx \FCap\NoCap
       \setbox\@tempboxa\hbox{#1}  
      \else
       \setbox\@tempboxa\hbox{#1: \small \it #2}
    \fi
    \ifdim \wd\@tempboxa >\hsize   
        \unhbox\@tempboxa\par      
      \else                        
        \hbox to\hsize{\hfil\box\@tempboxa\hfil}  
    \fi}

\def\@oddhead{\hbox{}\rightmark \hfil \rm\thepage}
\def\sectionmark#1{\markright {\sc{\ifnum \c@secnumdepth >\z@
      \S\thesection.\hskip 1em\relax \fi #1}}}

\catcode`\@=12

\def\oplabel#1{
  \def\OpArg{#1} \ifx \OpArg\Empty {} \else
  	\label{#1}
  \fi}

\newlength{\saveu}

\input{serg.tex}

\def\centeredepsfbox#1{\centerline{\epsfbox{#1}}}

\begin{document}

\title{Diversified homotopic behavior of closed orbits of some
$\rrrr$-covered Anosov flows}
\author{S\'{e}rgio R. Fenley
}
\maketitle

\vskip .2in

{\small{
\noindent
{\bf {Abstract}} $-$ 
We produce infinitely many examples of Anosov flows in closed $3$-manifolds
where the set of periodic orbits is partitioned into two infinite subsets.
In one subset every closed orbit is freely homotopic to infinitely 
other closed orbits of the flow. In the other subset every closed orbit
is freely homotopic to only one other closed orbit. 
The examples are obtained by Dehn surgery on geodesic flows. The manifolds
are toroidal and have Seifert  pieces and atoroidal pieces in their
torus decompositions.
}}

\vskip .2in
\section{Introduction}

This article deals with the question of free homotopies of closed orbits
of Anosov flows \cite{An} in $3$-manifolds. In particular we deal with the 
following question: how many closed orbits are freely homotopic to
a given closed orbit of the flow?
Suspension Anosov flows have the property that an arbitrary
closed orbit is not freely homotopic to any other closed orbit.
Geodesic flows have the property that every closed orbit (which 
corresponds to a geodesic in the surface) is only freely
homotopic to one other closed orbit. The other orbit corresponds
to the same geodesic in the surface, but being traversed in the
opposite direction. About twenty years ago the author proved that 
there is an infinite class of Anosov flows in closed hyperbolic
$3$-manifolds satisfying the property that every closed orbit
is freely homotopic to infinitely many other closed orbits
\cite{Fe1}.
Obviously this was diametrically opposite to the behavior of the 
previous two examples and it was also quite unexpected. 
This property in these examples is strongly connected with  the
large scale properties of Anosov flows when lifted to the universal cover.
In particular in these examples the property of infinitely orbits which
are  freely  homotopic to each other 
implies that the flows are not quasigeodesic, that is, orbits are not uniformly
efficient in measuring length in the universal cover \cite{Fe1}. This
provided the first examples of Anosov flows in hyperbolic $3$-manifolds
which are not quasigeodesic. 
The analysis of freely homotopic closed orbits of Anosov flows is also extremely
important in other situations, for example: \ 1) Analysing the interaction between incompressible
tori in the manifold and the Anosov flow \cite{Barb-Fe}, \ 2) Studying the structure of the Anosov flow when 
``restricted" to a Seifert piece of the torus decomposition of the manifold \cite{Barb-Fe}.

We define the {\em free homotopy class} of a closed orbit of a flow
to be the collection of closed orbits which are freely homotopic to the
original closed orbit. For an Anosov flow each free homotopy class
is at most infinite countable  as there are only countably many
closed orbits of the flow \cite{An}. In this article we are concerned with 
the cardinality of free homotopy classes. Suspensions have all free homotopy
classes with cardinality one and geodesic flows have all free homotopy
classes with cardinality two. In the hyperbolic examples mentioned above
every free homotopy class has infinite cardinality.
In addition if we do finite covers of geodesic flows, where we ``unroll the
fiber direction",  then we can get examples satisfying the property that every free homotopy class has
cardinality $2n$ where $n$ is a positive integer. 
The question we ask is whether we can have mixed behavior for an Anosov flow.
In other words,  can some free homotopy classes be infinite while others
have finite cardinality?
In this article we produce infinitely many examples where this indeed occurs.

\newpage
\noindent
{\bf {Main theorem}} $-$  There are infinitely many examples of Anosov flows
$\Phi$ in closed $3$-manifolds so that the set of closed orbits is partitioned
in two infinite subsets $A$ and $B$ so that the following happens.
Every closed orbit in $A$ has infinite free homotopy class.
Every closed orbit in $B$ has free homotopy class of cardinality two.
The examples are obtained by Dehn surgery on closed orbit of geodesic flows.
\vskip .1in

We thank the reviewer, whose suggestions greatly improved the presentation of this
article.

\section{Previous results and definitions}
\label{previous}

A three manifold $M$ is {\em irreducible} if every sphere bounds a ball \cite{He}. An 
{\em incompressible } torus is the image of an  embedding
$f: T^2 \rightarrow M$ which does not have compressing disks 
\cite{He}.  A $3$-manifold is {\em homotopically atoroidal} if
every $\pi_1$-injective map $f: T^2 \rightarrow M$ is homotopic into the 
boundary. The manifold is {\em geometrically atoroidal} if every incompressible
torus is homotopic to the boundary.
A $3$-manifold $M$ is {\em Seifert fibered} if it has a $1$-dimensional foliation
by circles \cite{He,Ep}. The {\em torus decomposition} states that every
compact, irreducible $3$-manifold $M$ can be decomposed by finitely
incompressible tori $T_1, ..., T_k$ so that the closure of every component
of $(M - \cup T_k)$ is either Seifert fibered or atoroidal 
\cite{Ja,Ja-Sh}.   A {\em graph manifold}
is an irreducible $3$-manifold whose pieces of the torus decomposition are all
Seifert fibered. 
The base space of a Seifert fibered space is the quotient of $M$ by the Seifert
fibration. It is a $2$-dimensional orbifold with finitely many cone points.
The Seifert space is called {\em small} if the base is either the disk with less
than 3 cone points or the sphere with less than 4 cone points.

If $M$ closed admits an Anosov flow then $M$ is irreducible \cite{Ro}. But $M$ may
be toroidal, for example this happens in the  case of geodesic flows or suspensions. In the case of
geodesic flows the whole manifold is Seifert fibered.

\subsection{Anosov flows and $\rrrr$-covered Anosov flows}

Let $\Phi$ be an Anosov flow in $M^3$ and let $\ls, \lu$ be its stable and unstable
foliations respectively. The leaves of $\ls, \lu$ can only be planes, annuli and
M\"{o}bius bands \cite{An}.   A leaf $L$ of $\ls$ or $\lu$ is an
 annulus or M\"{o}bius band if and only $L$ contains a closed orbit of $\Phi$ \cite{An}.
Let $\wls, \wlu$ be the lifted foliations to the universal cover $\mi$.
Let $F$ be a leaf of $\wls$ or $\wlu$. By the above, the leaf $F$ has non trivial
stabilizer if and only if $\pi(F)$ has a closed orbit of $\Phi$. Here the map $\pi: \mi \rightarrow M$
is the universal covering map. The stabilizer of $F$ is $\{ g \in \pi_1(M) \ | \ g(F) = F \}$.

\begin{theorem}{}{(\cite{Fe2,Fe3})}
Suppose that $\Phi$ is an Anosov flow in $M^3$ and suppose that $\alpha$ and $\beta$
are closed orbits of $\Phi$ which are freely homotopic to each other, as oriented curves.
Then there is $\gamma$ periodic orbit of $\Phi$ so that $\alpha$ is freely homotopic
to $\gamma^{-1}$ as oriented curves.
\label{inverse}
\end{theorem}

For an arbitrary Anosov flow  $\Phi$ the orbit
space $\oo$ of the lifted flow  $\wwp$ to the universal cover
 is homeomorphic to the plane $\rrrr^2$
\cite{Fe1}.
The lifted foliations $\wls, \wlu$ are invariant by the flow $\wwp$ so they induce 
one dimensional foliations
$\oos, \oou$ in $\oo$.

\begin{define}{}
A foliation $\fol$ in $M$  is $\rrrr$-covered if the leaf space of the lifted
foliation $\fn$ to the universal cover $\mi$ is homeomorphic to the real numbers $\rrrr$
\cite{Fe1}.
\end{define}

\begin{define}{} 
An Anosov flow is $\rrrr$-covered if its stable foliation
(or equivalently its unstable foliation \cite{Ba1,Fe1}) is $\rrrr$-covered.
\end{define}

Examples of $\rrrr$-covered Anosov flows are suspensions and geodesic flows \cite{Fe1}.
In \cite{Fe1} the author showed that there is an infinite class of $\rrrr$-covered Anosov flows 
on hyperbolic $3$-manifolds. Barbot \cite{Ba2} proved that the Handel-Thurston 
Anosov flows \cite{Ha-Th} are $\rrrr$-covered, as well as an infinite class of Anosov flows
in graph manifolds.
The Anosov flows in the Main theorem are $\rrrr$-covered. 
On the other hand the class of non $\rrrr$-covered Anosov flows is extremely large.
For example Barbot \cite{Ba1} proved that $\rrrr$-covered Anosov flows are transitive.
Hence the intransitive Anosov flows constructed by Franks and Williams 
\cite{Fr-Wi} are not $\rrrr$-covered.
In addition the Bonatti and Langevin examples \cite{Bo-La} are transitive Anosov flows
which are not covered. Barbot \cite{Ba2} constructed many other examples of transitive
Anosov flows in graph manifolds and these examples were greatly generalized in
\cite{Barb-Fe}. In all of these examples the underlying manifold is toroidal and
consequently not hyperbolic.

\subsection
{\bf {$\rrrr$-covered Anosov flows}}

Suppose that $\Phi$ is an $\rrrr$-covered Anosov flow.
Then  there are two possibilities 
for the topological structure of the lifted stable and unstable foliations $\wls, \wlu$
to $\mi$ \cite{Ba1,Fe1} or equivalently for the topological structure of the foliations
$\oos, \oou$ in $\oo \cong \rrrr^2$.

\begin{itemize} 

\item
Suppose that every leaf of $\wls$ intersects every leaf of $\wlu$.
In this case $\Phi$ is said to have  the {\em product} type.
Then Barbot \cite{Ba1} showed that $\Phi$ is topologically equivalent 
to a suspension Anosov flow. This implies that $M$ fibers over the
circle with fiber a torus. 

\item
The other possibility is that $\Phi$ is  a {\em skewed}
$\rrrr$-covered Anosov flow. This means that $\oo$ has a model homeomorphic
to an infinite strip $(0,1) \times \rrrr$. 
In addition the model  satisfies the following properties.
The stable foliation $\oos$ in $\oo$ 
is a foliation by horizontal segments in $(0,1) \times \rrrr$. The unstable foliation
is a foliation by parallel segments in $(0,1) \times \rrrr$ which make and angle
$\theta$ which is not $\pi/2$ with the horizontal. That is, they are not
vertical and hence an unstable leaf does not intersect every stable
leaf and vice versa. 
We refer to figure \ref{skew}.
Notice that every unstable leaf $u$ of $\oou$ intersects an interval $J_u$ of
stable leaves of $\oos$. This is a strict subset of the leaf space of $\oos$.
The model implies that different leaves $u$  of $\oou$ generate different intervals $J_u$.
\end{itemize}

In \cite{Fe3} the author proved that if $\Phi$ is a skewed $\rrrr$-covered Anosov flow
then the underlying manifold is orientable. With this one can also produce infinitely
many examples of transitive non $\rrrr$-covered Anosov flows where the underlying
manifold is hyperbolic.

\begin{proposition}{}{Proposition 3.1 of \cite{Ba-Fe}}
Let $\gamma$ be a closed orbit of an Anosov flow $\Phi$. Then any lift $\widetilde \gamma$
of $\gamma$ 
to $\mi$ is unknotted. In particular $\pi_1(\widetilde M - \widetilde \gamma)$ is
${\bf Z}$ and $\widetilde M - \widetilde \gamma$ is an open solid torus.
\label{unknot}
\end{proposition}

\subsection{Dehn surgery and Fried's flow Dehn surgery}

{\bf Dehn filling} -
Let $N$ be a $3$-manifold which is the interior of a compact $3$-manifold
$\hat{N}$ so that the boundary of $\hat{N}$ is a union of 
tori $P_1, ..., P_k$. Choose simple closed curve generators
$a_i, b_i$ for each $\pi_1(P_i)$. Let $N_{(x_1,y_1),...,(x_k,y_k)}$
be the manifold obtained by {\em Dehn filling} $N$ (or more specifically $\hat{N}$)
with $k$ solid tori $V_i$ as follows: for each $i$ glue the boundary of
the solid torus $V_i$ to $P_i$ by a homeomorphism so that the meridian
in $V_i$ is glued to the curve $x_i a_i + y_i b_i$ in $P_i$. We require
that the pair of integers $x_i, y_i$ are relatively prime to ensure that $x_i a_i + y_i b_i$
is a simple closed curve in $P_i$. The topological type of the Dehn filled manifold
is completely determined by the collection of pairs of integers $(x_i,y_i)$.
They are called the Dehn  surgery coefficients.

\vskip .1in
\noindent
{\bf Dehn surgery} -
Let $M$ be a  closed $3$-manifold and $\gamma$  an orientation preserving simple 
closed curve in $M$. Let $N(\gamma)$ be a solid torus neighborhood
of $\gamma$ and $M' = M - \mathring{N}(\gamma)$. Choose a pair
of generators for $\partial N(\gamma) \subset \partial M'$. Then {\em Dehn surgery}
on $\gamma$ with coefficients $x,y$ is the Dehn filled manifold
$M'_{(x,y)}$. 

\vskip .1in
\noindent
{\bf Remarks}  \ 1) More generally one can do Dehn surgery on $\gamma$ if $M$ has
boundary, or is not compact. \ 2) Dehn surgery is very general: any closed orientable $3$-manifold
can be obtained from the $3$-sphere by iterated Dehn surgery \cite{Li}.

\vskip .1in
\noindent
{\bf Hyperbolic Dehn surgery \cite{Th1,Th2,Be-Pe}} - Let $M$ be a complete hyperbolic manifold
with finite volume and not compact. Then $M$ is homeomorphic to the interior of a compact
manifold $\hat{M}$ with boundary components $P_1, ..., P_k$ which we assume are tori. 
Each such torus corresponds to a cusp in $M$.
As above one can do Dehn filling to obtain the closed manifold $M_{(x_1,y_1),..., (x_k,y_k)}$.
Thurston proved except for finitely many choices of the integers $x_i, y_i$ the manifold
$M_{(x_1,y_1),..., (x_k,y_k)}$ admits a hyperbolic structure. If one fills only one of the 
cusps, or a subset of the cusps,
 then for big enough surgery coefficients, 
 the resulting manifold is hyperbolic but not closed, it still has some cusps.
The reference \cite{Th1} has a very extensive and detailed proof of this in the case of the figure
eight knot complement in the sphere ${\bf S}^3$.  The book \cite{Be-Pe} has a proof
in the general case.

\vskip .1in
\noindent
{\bf Fried's flow Dehn surgery \cite{Fr}} -
Suppose that $\Phi$ is an Anosov flow and that $\gamma$ is a closed orbit which is 
orientation preserving in $M$, and so that the stable leaf $\ls(\gamma)$  of $\gamma$ is an annulus
(as opposed to a M\"{o}bius band). Let $N(\gamma)$ be a solid torus neighborhood of $\gamma$.
Choose generators for $\pi_1(P)$ where  $P = \partial (N(\gamma))$ as follows. Let $a$ be a meridian
in $N(\gamma)$ $-$ unique up to inverse in $\pi_1(P)$. Let $b$ be the intersection of the local stable
leaf of $\gamma$ with $\partial N(\gamma)$. 
Choose the orientation in $b$ to be the one induced by the positive flow direction in $\gamma$.
This is the ``longitude" in $\partial N(\gamma)$ 
in this case.
Do $(1,n)$ Dehn surgery on $\gamma$. Fried \cite{Fr} showed
how to do this along the flow: blow up the orbit $\gamma$ (producing a boundary torus). Then blow
down this boundary  torus to a closed curve according to the surgery coefficients $(1,n)$. 
The resulting flow is still an Anosov flow
and is denoted by $\Phi_{(1,n)}$. The orbit $\gamma$ blows up to a torus and then blows
down to an orbit  of $\Phi_{(1,n)}$. With this construction notice that there
is a bijection between the orbits of $\Phi$ and the orbits of the surgered flow
$\Phi_{(1,n)}$. 

\begin{theorem}{(Fe1)}{}
Let $\Phi$ be an $\rrrr$-covered Anosov flow in $M^3$ of skewed type. Let $\gamma$ be 
a closed orbit so that $\ls(\gamma)$ is an annulus. Then under a positivity condition, every
Anosov flow $\Phi_{(1,n)}$ is $\rrrr$-covered and of skewed type. Under appropriate
choices of the basis of $\pi_1(N(\gamma))$ the positivity condition is satisfied for
all positive $n$. In particular this is true if $\Phi$ is the geodesic flow in $T_1 S$ where
$S$ is a closed hyperbolic surface.
\label{rcov}
\end{theorem}

In general the positivity condition is satisfied either for all positive $n$ or for all
negative $n$ \cite{Fe1}. 
The issue is that the meridian is well undefined up to inverse.
Hence there are two possibilities for the basis, and one of them satisfies the positivity
condition for every positive $n$.
As it is not needed in this article we do not specify exactly when the positivity condition
holds.

\begin{theorem}{}{(\cite{Fe1})}
Let $M = T_1 S$ where $S$ is a closed, orientable hyperbolic surface and let $\Phi$ be
the geodesic flow in $M$. Let $\gamma$ be a closed geodesic in $S$ which fills
$S$. 
Let $\gamma_1$ be a periodic orbit of $\Phi$ which projects to $\gamma$ in $S$.
Do $(1,n)$ Dehn surgery on $\gamma_1$ satisfying the positivity condition to generate
manifold $M_s$ and Anosov flow $\Phi_s$. 
By theorem \ref{rcov} \  $\Phi_s$ is $\rrrr$-covered. 
Suppose that $n$ is big enough so that
$M_s$ is hyperbolic. 
Then 
 every closed orbit of $\Phi$ is freely homotopic to infinitely many other closed orbits.
\label{infinite}
\end{theorem}

\section{Atoroidal submanifolds of unit tangent bundles of surfaces}

Let $S$ be a closed hyperbolic surface and let $M = T_1 S$ be the unit tangent
bundle of $S$. Notice that $M$ is orientable, whether $S$ is orientable
or not. In the next section we will do Dehn surgery on a closed
orbit of the geodesic flow to obtain the examples of flows for our Main theorem.
We use the following notation to denote the projection map

$$\tau: \ M = T_1 S \ \rightarrow \ S$$

\noindent
which is the projection of a unit tangent vector to its basepoint in $S$.

Let $\Phi$ be the geodesic flow in $M$. It is well known that $\Phi$ is an Anosov flow \cite{An}.
Let $\alpha$ be a closed geodesic in $S$. This geodesic of $S$
generates two orbits of $\Phi$, let $\alpha_1$ be one such orbit.
This is equivalent to picking an orientation along $\alpha$.
Let $S_1$ be a subsurface of $S$ that $\alpha$ fills. If 
$\alpha$ is simple then $S_1$ is an annulus. If $\alpha$ fills
$S$ then $S_1 = S$.  Let $S_2$ be the closure in $S$ of 
$S - S_1$ $-$ which may be empty.
Let $M_i = T_1 S_i, \ i = 1,2$. Notice that both $M_1$ and $M_2$ are Seifert fibered (with boundary).

The purpose of this section is to prove the following result.

\begin{proposition}{(atoroidal)}{}
The submanifold \  $M_1 - \alpha_1$ \ is homotopically atoroidal.
\label{ator}
\end{proposition}

\begin{proof}{}
We will prove that $M_1 - \alpha_1$ is geometrically atoroidal.
This statement means the following: notice that $M_1 - \alpha_1$ is 
not compact, but $M_1 - \mathring{N}(\alpha_1)$ is compact. The statement
means that $M_1 - \mathring{N}(\alpha_1)$ is homotopically atoroidal.
Notice that $M_1 - \alpha_1$ is irreducible \cite{He}.
Gabai \cite{Ga} showed that since $M_1 - \alpha_1$ is not
a small Seifert fibered space  then
$M_1 - \alpha_1$ is also homotopically atoroidal.

Let $T$ be an incompressible torus in $M_1 - \alpha_1$. 
We think of $T$ as contained in $M$.
There are $2$ possibilities:

\vskip .15in
\noindent
{\underline {Case 1}} $-$ $T$ is $\pi_1$-injective in $M$.

Here $T$ is contained in $M - M_2$. We use that $M$ is Seifert fibered.
In addition $T$ is incompressible, so it is an essential lamination in $M$ \cite{Ga-Oe}.
By Brittenham's theorem \cite{Br} $T$ is isotopic to either a vertical
torus or a horizontal torus in $M$. Vertical torus means it is a union
of $S^1$ fibers of the Seifert fibration. Horizontal torus means
that it is transverse to these fibers. Since $S$ is a hyperbolic
surface, there is no horizontal torus in $M = T_1 S$. It follows that $T$ is
isotopic to a vertical torus $T'$. In addition since $T$ itself is disjoint
from $M_2$ and $M_2$ is saturated by the Seifert fibration, we can push
the isotopy away from $M_2$ and suppose it is contained in $M_1$.
Finally the isotopy forces an isotopy of the orbit $\alpha_1$ into
a curve $\alpha'$ disjoint from $T'$. 
This isotopy projects by $\tau$ to an homotopy in $S$ from $\alpha$ to
a curve $\alpha^*$ and the image of this homotopy is contained in 
$S_1$, since the isotopy in $M$ has image in $M_1$.
The curve $\alpha^*$ is disjoint from the projection $\tau(T')$.
Since $T'$ is vertical  this projection is a simple
closed curve $\beta$ in $S_1$. Since $\alpha$ fills $S_1$, $\alpha^*$ is
homotopic to $\alpha$ in $S$, and $\beta$ is disjoint from
$\alpha^*$,  it now
follows that $\beta$ is a peripheral curve in $S_1$.
By another isotopy we can assume that $\beta$ does not intersect
$\alpha$ or that $T'$ does not intersect $\alpha_1$.

In addition the isotopy from $T$ to $T'$ can be extended to an isotopy
from $M_1$ to itself. The geometric intersection number of $T, T'$
with $\alpha_1$ is zero. So we can adjust the isotopy  so that 
the images of $T$ under the isotopy never intersect $\alpha_1$, and consequently
we can further adjust it so that it leaves $\alpha_1$ fixed
pointwise. 
In other words this induces an isotopy in $M_1 - \alpha_1$ from
$T$ to $T'$. This shows that $T$ is peripheral
in $M_1 - \alpha_1$.
This finishes the proof in this case.

\vskip .15in
\noindent
{\underline {Case 2}} $-$ $T$ is not $\pi_1$-injective in $M$.

In particular since $T$ is two sided (as $M$ is orientable), then
$T$ is compressible \cite{He}. This means that there is a closed disk
$D$ which compresses $T$ \cite{He}, chapter 6. 
Since $T$ is incompressible in $M_1 - \alpha_1$,
then $D$ intersects $\alpha_1$. 
Let $D_1, D_2$ be parallel isotopic copies of $D$ very near $D$ which 
also are compressing disks for $T$. Then $D_1, D_2$ intersect
$T$ in two curves which partition $T$ into two annuli.
One annulus is very near both $D_1$ and $D_2$, we call it $A_1$, let
$A$ be the other annulus which is almost all of $T$. Then $A \cup D_1 \cup D_2$
is an embedded two dimensional sphere $W$. Since $M$ is irreducible
then $W$ bounds a $3$-ball $B$. 
There are two possibilities for the
sphere $W$ and ball $B$. In addition $A_2 \cup D_1 \cup D_2$ also
obviously bounds a ball $B_1$ which is very near the disk $D$.

\vskip .05in
Suppose first that the ball $B$ contains the torus $T$. This means that $A_2$ and
consequently also $B_1$, are both contained in $B$. In addition $B_1$ 
is a regular tubular neighborhood of a properly
embedded arc $\gamma$ in $B$. 
The intersection of $\alpha_1$ with $B_1$  is a collection of arcs 
which are isotopic to the core $\gamma$ of $B_1$. Let $\delta_1$ be one
such arc. By Proposition \ref{unknot} flow lines of Anosov flows
lift to unknotted curves in $\mi$. This implies  that $\gamma$ is 
unknotted in $B$ and also 
implies that $\pi_1(B - \gamma)$
is ${\bf Z}$. In particular the torus $T$ is compressible in $B - B_1$, that is,
the closure of $B - B_1$ is a solid torus. It follows that the 
$T$ is compressible in $M_1 - \alpha_1$. 
This contradicts the assumption that $T$ is incompressible in $M_1 - \alpha_1$.

\vskip .05in
The second possibility is that the ball $B$ does not contain $T$. 
In particular $B$ and $B_1$ have disjoint interiors and the
the union $B \cup B_1$ is a solid torus $V$ with boundary $T$.
The union of $B$ and $B_1$ cannot be a solid Klein bottle because 
$M$ is orientable. 
This solid torus lifts to an infinite solid tube $\widetilde V$ in $\mi$
with boundary $\widetilde T$ which is an infinite cylinder.
Notice that there is a lift $\widetilde \alpha_1$ of $\alpha_1$
contained in $\widetilde V$ so $\widetilde T$ cannot be compact.
Again by the result of Proposition \ref{unknot}, the infinite curve
$\widetilde \alpha_1$ is unknotted in $\mi$ and hence it is isotopic to
the core of $\widetilde V$. 

Let $\beta$ be a simple closed curve in $V$ which is isotopic to the core of $V$.
If $\alpha_1$ is not isotopic to $\beta$ then it is homotopic to a power
$\beta^n$ where $n > 1$. Projecting $\beta$ to $\tau(\beta)$ in $S$ we obtain
a closed curve in $S$ so that $(\tau(\beta))^n$ is freely homotopic to $\alpha$.
But $\alpha$ is an {\underline {indivisible}} closed geodesic and
represents an indivisible element of $\pi_1(S)$. It follows that this cannot happen.
We conclude that $\alpha_1$ is isotopic to the core of $V$.
It follows that $T$ is isotopic to the boundary of a regular 
neighborhood of $\alpha_1$ in $M_1 - \alpha$ and hence again $T$ is peripheral
in $M_1 - \alpha_1$.

This finishes the proof of proposition \ref{ator}.
\end{proof}

\noindent
{\bf {Remark}} In the case that $\alpha$ fills $S$ Proposition \ref{ator} is well
known and there is a written proof by Foulon and Hasselblatt in \cite{Fo-Ha}.

Since $M_1 - \alpha_1$ is atoroidal  the geometrization theorem in the Haken
case \cite{Th1,Th2} shows that $M_1 - \alpha_1$ admits a hyperbolic
structure. The hyperbolic Dehn surgery
theorem of Thurston implies that for almost all Dehn fillings
along $\alpha_1$, the resulting manifold $M_s$ is hyperbolic. Notice
that since $M_1$ has boundary, the statement $M_s$ is hyperbolic
means that the interior of $M_s$ has a complete hyperbolic structure
of finite volume, and each (torus) component of $M_1$ generates
a cusp in the hyperbolic structure in the interior of $M_s$.

\section{Diversified homotopic behavior of closed orbits}

First we prove the statements about free homotopy classes of suspension
Anosov flows and geodesic flows mentioned in the introduction.
Suppose first that $\Phi$ is the geodesic flow in $M = T_1 S$, where
$S$ is a closed, orientable hyperbolic surface. Suppose that $\alpha, \beta$
are closed orbits of $\Phi$ which are freely homotopic to each other in $M$. Then the
projections $\tau(\alpha), \tau(\beta)$ of these orbits to the surface
$S$ are freely homotopic in $S$. But $\tau(\alpha), \tau(\beta)$ are 
closed geodesics in a hyperbolic surface, so they are freely homotopic
if and only if they are the same geodesic. If $\alpha$ and $\beta$ are
distinct, this can only happen if they represent the same geodesic
$\tau(\alpha)$ of $S$  which is being traversed in opposite directions.
Conversely if $\tau(\alpha) = \tau(\beta)$ and they are traversed
in opposite directions, there is a free homotopy from $\alpha$ 
to $\beta$. This is achieved by considering all unit tangent vectors
to $\tau(\alpha)$ in the direction of $\alpha$ and then
at time $t$, $0 \leq t \leq 1$, rotating all these vectors
by an angle of $t \pi$. At $t = \pi$ we obtain the tangent vectors
to $\tau(\alpha)$ pointing in the opposite direction, that is,
the direction of $\beta$. This shows that every free homotopic
class of the geodesic flow has exactly two elements.
The orientability of $S$ is used because if $S$ is not orientable and 
$\tau(\alpha)$ is an orientation reversing closed geodesic, one cannot
continuously turn the angle along $\tau(\alpha)$.

Now consider a suspension Anosov flow $\Phi$. By Theorem \ref{inverse},
given an arbitrary Anosov flow which admits freely homotopic closed orbits, then
the following happens. There are closed
orbits $\alpha$ and $\beta$ so that $\alpha$ is freely homotopic
to $\beta^{-1}$ as oriented periodic orbits. For
suspension Anosov flows this is a problem as follows. This is because
there is a cross section $W$ which intersects all orbits of $\Phi$.
Suppose that the algebraic intersection number of $\alpha$ and $W$ is positive.
Then since $\alpha$ is freely homotopic to $\beta^{-1}$ it follows
that the algebraic intersection number of $\beta$ and $W$ is negative.
But this is impossible as $W$ is a cross section and transverse to $\Phi$.
This shows that every free homotopy class of a suspension is a singleton.
Another proof of this fact is the following. There is a path metric in 
$M$ which comes from a Riemannian metric in the universal
cover $\mi \cong \rrrr^3$ with coordinates $(x,y,t)$ given by
the formula $ds^2 \ = \ \lambda^{2t}_1 dx^2 + \lambda^{-2t}_2 dy^2
+ dt^2    \ \ (1)$, where $\lambda_1, \lambda_2$ are real numbers $> 1$.
The lifted flow $\wwp$ has formula $\wwp_t(x,y,t_0) = 
(x,y, t_0 + t)  \ \ (2)$. \
If $\alpha, \beta$ are freely homotopic closed orbits of $\wwp$, then 
they lift to two distinct orbits of $\wwp$  which are a bounded distance from each other.
But formulas (1) and (2) show that no two distinct entire orbits of $\wwp$ are
a bounded distance from each other. This also shows that free homotopy
classes are singletons.

For the property of infinite free homotopy classes for the examples in hyperbolic
$3$-manifolds see Theorem \ref{infinite}.

\vskip .15in
We now proceed with the construction of  the examples  with diversified homotopic 
behavior and we prove the Main theorem.

Let $S$ be a hyperbolic surface and $\alpha$ a closed geodesic
that does not fill $S$. As in the previous section let $S_1$ be a 
subsurface that $\alpha$ fills and let $S_2$ be the closure of
$S - S_1$. Let $M = T_1 S$ and $\Phi$ the geodesic flow 
of $S$ in $M$. Let $\alpha_1$ be an orbit of $\Phi$ so that $\tau(\alpha_1)
= \alpha$. Let $M_i = T_1 S_i, \  i = 1, 2$.
In the previous section we proved that $M_1 - \alpha_1$ is atoroidal.

Now we will do Fried's  Dehn surgery on $\alpha_1$. 
For simplicity we will assume that the unstable foliation of $\Phi$
 (or equivalently the stable foliation of $\Phi$)
is transversely orientable. This is equivalent to the surface
$S$ being orientable.
In particular this implies that the stable leaf of $\alpha_1$ is
a annulus. Let $Z$ be the boundary of a small tubular solid torus
neighborhood $Z_0$ of $\alpha_1$ contained in $M_1$. 
Then $Z$ is a two dimensional torus and we will choose a base for
$\pi_1(Z) = H_1(Z)$. We assume that $Z$ is transverse to the local
sheet of the stable leaf of $\alpha_1$. Then this local sheet intersects
$Z$ in a pair of simple closed curves. Each of these defines a 
longitude $(0,1)$ in $\pi_1(Z)$, choose the direction which is
isotopic to the flow forward direction along $\alpha_1$.
The boundary of a meridian disk in $Z_0$ defines the meridian
curve $(0,1)$ in $\pi_1(Z)$. The meridian is well defined up to sign.
If the stable foliation of $\Phi$  were not transversely orientable
and $\alpha$ were an orientation reversing curve, then the stable leaf
of $\alpha$
would be a M\"{o}bius band and the intersection of the local sheet with
$Z$ would be a single closed curve. This closed curve would intersect
the meridian twice and could not form a basis of $H_1(Z)$ jointly
with the meridian. We do not want that, hence one of the reasons to
restrict to $S$ orientable.

Now we perform Fried's Dehn surgery on $\alpha_1$ \cite{Fr} as described in section
\ref{previous}.
We do $(1,n)$ surgery on $\alpha_1$, so that the following happens.
The resulting flow is Anosov in the Dehn surgery manifold
$M_{\alpha}$. The meridian is chosen so that  for any $n > 0$ the Dehn surgery flow $\Phi_{\alpha}$
with new meridian the $(1,n)$ curve is an $\rrrr$-covered Anosov flow.
Recall that there is a  bijection between the orbits of the surgered flow
$\Phi_{\alpha}$ and the orbits of the original flow $\Phi$. 
Given an orbit $\gamma$ of $\Phi_{\alpha}$ we let $\gamma'$ be the
corresponding orbit of $\Phi$ under this bijection.

We are now ready to prove the prove the main result of this article, which is
restated with more detail below.

\begin{theorem}{(diversified homotopic behavior)}{}
Let $S$ be an orientable, closed hyperbolic surface with a closed geodesic
$\alpha$ which does not fill $S$. Let $S_1$ be a subsurface of $S$ which
is filled by $\alpha$ and let $S_2$ be the closure of $S - S_1$.
We assume also that $S_2$ is not a union of annuli.
Let $M = T_1 S$ with geodesic flow $\Phi$ and let $M_i = T_1 S_i, \  i = 1,2$. 
Let $\alpha_1$ be a closed orbit of $\Phi$ which projects to $\alpha$ in $S$.
Do $(1,n)$ Fried's Dehn surgery along $\alpha_1$ to yield 
a manifold $M_{\alpha}$ and an Anosov flow $\Phi_{\alpha}$ so that
$\Phi_{\alpha}$ is $\rrrr$-covered.
Since $M_2$ is disjoint from $\alpha_1$ it is unaffected by the Dehn surgery
and we consider it also as a submanifold of $M_{\alpha}$. Let $M_3$ be
the closure of $M_{\alpha} - M_2$. We still denote by $\alpha_1$ the orbit
of $\Phi_{\alpha}$ corresponding to $\alpha_1$ orbit of $\Phi$.
Proposition \ref{ator} implies that $M_3 - \alpha_1$ is atoroidal and
for $n$ big the hyperbolic Dehn surgery theorem \cite{Th1,Th2} implies that
$M_3$ is hyperbolic. Choose one such $n$. 
Consider the bijection $\beta \rightarrow \beta'$ between closed orbits
of $\Phi_{\alpha}$ and those of $\Phi$.
Then the following happens:

\begin{itemize}

\item i) Let $\gamma$ be a closed orbit of $\Phi_{\alpha}$ so that the 
corresponding orbit $\gamma'$ of $\Phi$ is homotopic into the 
submanifold $M_2$. 
Equivalently $\gamma'$ projects to a geodesic in $S$ which is disjoint
from $\alpha$ in $S$. Then $\gamma$ is freely homotopic in $M_{\alpha}$
to just one
other closed orbit of $\Phi_{\alpha}$.

\item
ii) Let $\gamma$ be a closed orbit of $\Phi_{\alpha}$ which corresponds to a closed
orbit $\gamma'$ of $\Phi$ which is not homotopic into $M_2$. Equivalently
$\gamma'$ projects to a geodesic in $S$ which transversely intersects $\alpha$.
Then $\gamma$ is freely homotopic in $M_{\alpha}$ to infinitely many other closed orbits
of $\Phi_{\alpha}$. 

\item
In addition both classes i) and ii) have infinitely many elements.

\end{itemize}
\label{main}
\end{theorem}

\begin{proof}{}
First we prove that both classes i) and ii) are infinite.
Since orbits of $\Phi_{\alpha}$ are in one to one correspondence
with orbits of $\Phi$,   one can think of these as  statements about closed orbits
of $\Phi$.
Any closed geodesic of $S$ which intersects $\alpha$ is in class ii).
Clearly there are infinitely many such geodesics so class ii) is infinite.
On the other hand since $S_2$ is not a union of annuli, there is
a component $S'$ which is not an annulus. Any geodesic $\beta$ of
$S$ which is homotopic into $S'$ creates an orbit in class i). 
Since $S'$ is not an annulus, there are infinitely many such geodesics
$\beta$. This proves that i) and ii) are infinite subsets.

An orbit $\delta$ of $\Phi$ which projects in $S$ to a geodesic intersecting
$\alpha$ cannot be homotopic into $M_2$.
Otherwise the homotopy projects in $S$ 
to an homotopy from a geodesic intersecting
$\alpha$ to a curve in $S_2$ and hence to a geodesic
not intersecting $\alpha$.  This is impossible as closed geodesics in hyperbolic
surfaces intersect minimally.
Conversely if an orbit $\delta$ projects to a geodesic not intersecting
$\alpha$, then this geodesic is homotopic to a geodesic contained
in $S_2$. This homotopy lifts to a homotopy in $M$ from $\delta$ to a curve
in $M_2$. This proves the equivalence of the first 2 statements in i) 
and in ii).

\vskip .1in
Now we prove that conditions i), ii) imply the respective conclusions
about the size of the free homotopy classes.
Let $\widetilde \Phi_{\alpha}$ be the lifted flow to the universal
cover $\mi_{\alpha}$. 

The flow $\Phi_{\alpha}$ is $\rrrr$-covered. As explained in section \ref{previous}
there are two possibilites for $\Phi_{\alpha}$, either product or skewed.
If $\Phi_{\alpha}$ is product then $M_{\alpha}$ fibers over the circle with
fiber a torus. But in our case, $M_{\alpha}$ has a torus decomposition
with one hyperbolic piece $M_3$ and one Seifert piece $M_2$. 
Therefore it cannot fiber over the circle with fiber a torus. We conclude
that this case cannot happen.

\vskip .1in
Therefore $\Phi_{\alpha}$ is skewed.

Let then $\beta_0$ be a closed orbit of $\Phi_{\alpha}$.
since $\Phi_{\alpha}$ is an
skewed $\rrrr$-covered Anosov flow we will produce orbits $\beta_i, \  i \in {\bf Z}$
which are all freely homotopic to $\beta_0$. However it is not  a priori true
that all the orbits $\beta_i$ are distinct from each other, this will be analysed later.

Here we identify the fundamental group of the manifold with the set of 
covering translations of the universal cover.

\begin{figure}
\centeredepsfbox{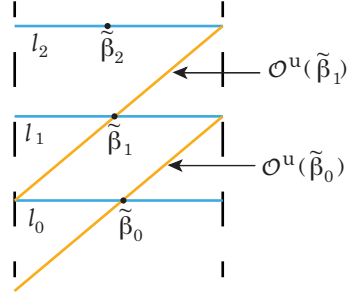}
\caption{The picture of the orbit space $\oo \cong (0,1) \times \rrrr$ of a skewed $\rrrr$-covered 
Anosov flow. The stable foliation $\oos$ is the foliation by horizontal
segments in $(0,1) \times \rrrr$. The unstable foliation $\oou$ is the foliation by
the parallel slanted segments.
This picture also shows how to construct the orbits
$\widetilde \beta_i$ starting with the orbit $\widetilde \beta_0$.}
\label{skew}
\end{figure}

\vskip .1in
\noindent
{\bf {Construction of the  orbits $\widetilde \beta_j$.}}

Lift $\beta_0$ to an orbit $\widetilde \beta_0$ contained in a stable leaf $l_0$ of $\oos$.
Let $g$ be the deck transformation of $\mi_{\alpha}$ which corresponds
to $\beta_0$ in the sense that it generates the stabilizer of $\widetilde \beta_0$.
Then $u = \oou(\widetilde \beta_0)$ intersects an open interval $J_u$ of stable leaves.
This is a strict subset of the leaf space of $\oos$ (equal to leaf space of $\wls$)
by the skewed property.
Let $l_1$ be one of the two stable leaves in the boundary of this interval.
The fact that there are exactly two boundary leaves in this interval is a direct
consequence of the fact that $\oos$ (or $\wls$) has leaf space $\rrrr$ and this fact is not
true in general. Since $g(\oou(\widetilde \beta_0)) = \oou(\widetilde \beta_0)$ and
$g$ preserves the orientation of $\oos$ (because $\ls$ is transversely orientable),
then $g(l_1) = l_1$. But this implies that there is an orbit $\widetilde \beta_1$
of $\wwp_{\alpha}$ in $l_1$ so that $g(\widetilde \beta_1) = \widetilde \beta_1$.
We refer to fig. \ref{skew} which shows how to obtain leaf $l_1$ and hence the 
orbit $\widetilde \beta_1$.
This orbit projects to a closed orbit $\beta_1$ of $\Phi_{\alpha}$ in $M_{\alpha}$.
Since both are associated to $g$, it follows that $\beta_0, \beta_1$ are freely
homotopic. More specifically if we care about orientations then the positively
oriented orbit $\beta_0$ is freely homotopic to the inverse of the positively
oriented orbit $\beta_1$. 

\vskip .1in
\noindent
{\bf {Remark}} $-$ Transverse orientability of $\ls$ is necessary for this.
If for example $\ls$ were not transversely orientable and
the unstable leaf of $\beta_0$ were a M\"{o}bius band then the transformation
$g$ as constructed above {\underline {does not preserve}} the leaf $l_1$ as constructed
above. Therefore $\beta_0$ is not freely homotopic to $\beta_1$ as unoriented curves.
But $g^2$ preserves $l_1$ and from this it follows that the square
$\beta_0^2$ (as a non simple closed curve) is freely homotopic to
$\beta^{2}_1$. 
\vskip .1in

We proceed with the construction of freely homotopic orbits of $\Phi_{\alpha}$.
From now on we  iterate the procedure above: use $\oou(\widetilde \beta_1)$ to
produce a leaf $l_2$ of $\oos$ invariant by $g$, and a closed orbit $\beta_2$ freely homotopic
to $\beta_1$ $-$ if we consider them just as simple closed curves.
We again refer to fig. \ref{skew}.
 Now iterate
and produce
$\widetilde \beta_i, i \in {\bf Z}$ orbits of $\wwp_{\alpha}$ so that they are all invariant
under $g$ and project to closed orbits $\beta_i$ of $\Phi_{\alpha}$ which are
all freely homotopic to $\beta_0$ as unoriented curves.

\vskip .1in
\noindent
{\bf {Orbits freely homotopic to $\beta_0$.}}

The covering translation $g$ preserves the leaf $\oou(\widetilde \beta_0)$.
Since $\beta_0$ is the only periodic orbit in $\lu(\beta_0)$, it follows
that $g$ only preserves the orbit $\widetilde \beta_0$ in $\oou(\widetilde \beta_0)$.
Therefore $g$ does not leave invariant any 
{\underline {stable}} leaf between $\oos(\widetilde \beta_0)$
and $\oos(\widetilde \beta_1)$ and similarly $g$ does not leave invariant
any stable leaf between $\oos(\widetilde \beta_i)$ and $\oos(\widetilde \beta_{i+1})$
for any $i \in {\bf Z}$.
It follows that the collection $\{ \oos(\widetilde \beta_i), i \in {\bf Z} \}$ is 
exactly the collection of stable leaves left invariant by $g$.

Suppose now that $\delta$ is an orbit of $\Phi_{\alpha}$ which is freely homotopic
to $\beta_0$. We can lift the free homotopy so that $\beta_0$ lifts to
$\widetilde \beta_0$ and $\delta$ lifts to $\widetilde \delta$. In particular
$g$ leaves invariant $\widetilde \delta$ and hence leaves
invariant $\oos(\widetilde \delta)$. It follows that 
$\oos(\widetilde \delta) = \oos(\widetilde \beta_i)$ for some $i \in {\bf Z}$.
As a consequence $\widetilde \delta = \widetilde \beta_i$ for 
$\widetilde \beta_i$ is the only orbit of $\wwp_{\alpha}$ left invariant
by $g$ in $\oos(\widetilde \beta_i)$. It follows that $\delta$ is one of 
$\{ \beta_j,  \ j \in {\bf Z} \}$.

\vskip .12in
\noindent
{\underline {Conclusion}} $-$ The free homotopy class of $\beta_0$ is finite if and
only if the collection $\{ \beta_i, i \in {\bf Z} \}$ is finite.
\vskip .08in

Suppose now that $\beta_i = \beta_j$ for some $i, j$ distinct.
Hence there is $f \in \pi_1(M_{\alpha})$ with $f(\widetilde \beta_i) = \widetilde \beta_j$.
Then $f$ sends $\oou(\widetilde \beta_i)$ to $\oou(\widetilde \beta_j)$.
By the definition of $\widetilde \beta_{i+1}$ it follows that 
$f$ sends $\widetilde \beta_{i+1}$ to $\widetilde \beta_{j+1}$. 
Iterating this procedure shows that
$f$ preserves the collection $\{ \widetilde \beta_k, \  k  \in {\bf Z} \}$.
In addition it follows easily that 
$f$ sends $\widetilde \beta_0$ to $\widetilde \beta_k$ for  $k = j-i$.
The free homotopy from $\beta_0$ to $\beta_k = \beta_0$ produces a $\pi_1$-injective
map of either the torus or the Klein bottle into $M$. 
We have to consider the Klein bottle because the free homotopy may be from
$\beta_0$ to the inverse of $\beta_0$ when we account for orientations along orbits.
Taking the square of this free homotopy if necessary we produce a $\pi_1$-injective
map of the torus into $M$. The torus theorem \cite{Ja,Ja-Sh} shows that the
free homotopy is homotopic into a Seifert piece of the torus decomposition
of $M_{\alpha}$. Therefore in our situation the homotopy is freely homotopic into $M_2$. 
It follows that the orbit $\beta'_0$ of $\Phi$ associated to $\beta_0$
is freely homotopic into $M_2$. Therefore the geodesic $\tau(\beta'_0)$
of $S$ does not intersect $\alpha$. 

This proves part ii) of the theorem: \  If the geodesic $\tau(\beta'_0)$ 
intersects $\alpha$ then the orbit $\beta_0$ of $\Phi_{\alpha}$
 is freely homotopic to infinitely many
other closed orbits of $\Phi_{\alpha}$.

\vskip .1in
Consider now a closed orbit $\beta_0$ of $\Phi_{\alpha}$ so that
it corresponds to a geodesic in $S$ which does not intersect $\alpha$.
This geodesic is $\tau(\beta'_0)$ which we denoted by $\gamma$. There
is a non trivial free homotopy in $M = T_1 S$ from $\beta'_0$ to itself
with the same orientation, obtained by turning the angle along $\gamma$
by a full turn, from $0$ to $2 \pi$.  Notice that this free homotopy
at some point is exactly $\beta'_0$ and at another point
it is exactly the orbit corresponding to the geodesic $\gamma$
being traversed in the opposite direction.
This free homotopy is entirely
contained in $M_2$ and therefore this free homotopy survives in 
the Dehn surgered manifold $M_{\alpha}$.
 By construction the image of the free homotopy in $M_{\alpha}$ contains
two distinct closed orbits of $\Phi_{\alpha}$, one of which is $\beta_0$.
In particular the free homotopy class of $\beta_0$ has at least two elements.
In addition the free homotopy produces a $\pi_1$-injective map from
$T^2$ into $M$. Choose a basis $g, f$ for $\pi_1(T^2)$ (seen as covering
translations in $\mi$) so that $g$ leaves invariant a lift $\widetilde\beta_0$
of $\beta_0$. Then 

$$g f(\widetilde \beta_0) \ \ = \ \ f g(\widetilde \beta_0) \ \ = \  \ f(\widetilde \beta_0).$$

\noindent
So $g$ also leaves invariant $f(\widetilde \beta_0)$. As seen in the paragraphs
``Orbits freely homotopic to $\beta_0$", it follows that $f(\widetilde \beta_0) = \widetilde \beta_j$
for some $j$ in ${\bf Z}$. This implies that the free homotopy class
of $\beta_0$ is finite.

Let now  $\delta$ be a closed orbit of $\Phi_{\alpha}$ which is freely homotopic
to $\beta_0$. 
In particular the free homotopy class of $\delta$ is the same as the free
homotopy class of $\beta_0$ and in the part entitled 
``Orbits freely homotopic to $\beta_0$" we showed that this free homotopy class
 is finite in this case.
In addition  from what  we already 
proved in the theorem, it follows that $\delta$ is isotopic into $M_2$ and
choosing $M_2$ appropriately we can assume that $\beta_0, \delta$ are
contained in $M_2$. 
Let the free homotopy from $\beta_0$ to $\delta$ be realized by
a $\pi_1$-injective annulus $A$ which is in general position.
The annulus $A$ is a priori only immersed. Let $T = \partial M_3 
= \partial M_2$ an embedded torus in $M_{\alpha}$ which is
$\pi_1$-injective. Put $A$ in general position with respect to $T$ and
analyse the self intersections. Any component which is null homotopic
in $T$ can be homotoped away because $M_{\alpha}$ is
irreducible \cite{He,Ja}. After this is eliminated each component of $A - T$
is an a priori only immersed annulus. But since $M_3$ is a hyperbolic 
manifold with a single boundary torus $T$ it follows that $M_3$
is acylindrical \cite{Th1,Th2}. This means that any $\pi_1$-injective
properly immersed annulus is homotopic rel boundary into the boundary.
This is because parabolic subgroups of the fundamental group of $M_3$
$-$ as a Kleinian group, have an associated  maximal  ${\bf Z}^2$
parabolic subgroup \cite{Th1,Th2}. In particular this implies that
the annulus $A$ can be homotoped away from $M_3$ to be entirely
contained in $M_2$.
Therefore the free homotopy represented by the annulus $A$
survives if we undo the Dehn surgery on $\alpha$. This produces
a free homotopy between $\beta'_0$ and $\delta'$ in $M = T_1 S$.
But the free homotopy classes of geodesic flows all have exactly
two elements.
Therefore there is only one possibility for $\delta$ if $\delta$ is
distinct from $\beta_0$.
This shows that the free homotopy class of $\beta_0$ has exactly two
elements.


This finishes the proof of theorem \ref{main}
\end{proof}

\section{Generalizations}

There are a few ways to generalize the main result of this article. Here we mention two of them.

\vskip .1in
\noindent
{\bf {1) Finite covers and Dehn surgery}}

Let $M = T_1 S$ where $S$ is a closed orientable surface. Let $\Phi$ be the geodesic flow
in $M$. First take a finite cover of order $n$ of $M$ unrolling the circle fibers. Let this be
the manifold $M_1$ with lifted Anosov flow $\Phi_1$. Then every closed orbit 
of $\Phi_1$  is freely homotopic to $2n-1$ other closed 
orbits of the flow. The flow $\Phi_1$ is a skewed $\rrrr$-covered Anosov flow.
We use the covering map $\eta: M_1 \rightarrow M$ and the projection $\tau: M \rightarrow S$.
We do Dehn surgery on closed orbits of $\Phi_1$. Essentially the same proof as the Main
theorem yields the following result.

\begin{theorem}{}{} Let $S$ be a closed orientable surface and $M = T_1 S$
with Anosov flow $\Phi$. Let $M_1$ be a finite cover of $M$ where we unroll the Seifert fibers
and let $\Phi_1$ be the lifted flow to $M_1$.
Do $(1,n)$ Fried's flow  Dehn surgery on a closed orbit $\gamma$ of the $\Phi_1$
 so $\tau \circ \eta(\gamma)$ is a closed geodesic in $S$ which does not fill $S$
and some complementary component of $\tau \circ \eta(\gamma)$ in $S$  is not an annulus 
and so that the surgery satisfies the positivity condition. Let $\Phi_s$ be the resulting Anosov
flow in the surgery manifold.
Given  $\beta$ a closed orbit of $\Phi_s$, it  has an associated unique orbit of $\Phi_1$ which
in turn
projects under $\tau \circ \eta$  to a closed geodesic $\beta'$ of $S$.
Then i) If $\beta'$ does not intersect $\tau \circ \eta(\gamma)$ it follows $\beta$ is freely
homotopic to exactly $2n -1$ other orbits of $\Phi_s$. In addition ii) If $\beta'$ intersects
$\tau \circ \eta(\gamma)$ then $\beta$ is freely homotopic to infinitely many other closed
orbits of $\Phi_s$.
\end{theorem}

\vskip .1in
\noindent
{\bf {2) Dehn surgery on more than one closed orbit}}

In essentially the same way as in the proof of the Main theorem
We obtain a result similar to the Main theorem under Dehn surgery on finitely many closed orbits
as follows.

\begin{theorem}{}{}
Let $S$ be a closed orientable surface and $M = T_1 S$ with geodesic flow $\Phi$.
Let \ $\{ \alpha_i, \ \ 1 \leq i \leq i_0 \}$ be a finite collection of disjoint closed geodesics
in $S$ which are pairwise 
disjoint and some component of the complement of their union is not an annulus.
Let $\gamma_i$ be a closed orbit of $\Phi$ which projects to $\alpha_i$ in $S$.
For each $i$ do $(1,n_i)$  Fried's flow  Dehn surgery on $\gamma_i$ to yield an
Anosov flow $\Phi_s$ in the Dehn surgery manifold $M_s$ and so that the surgery
satisfies the positivity condition. Then $\Phi_1$ is an 
$\rrrr$-covered Anosov flow. There is a bijection between orbits of $\Phi_1$ and
orbits of $\Phi$. Given an orbit $\beta$ of $\Phi_s$ consider the orbit of $\Phi$
associated to it and project it to a closed geodesic $\beta'$ in $S$. Then
the following happens.
\ i) If $\beta'$ is disjoint from the union of the $\{ \alpha_i \}$ then $\beta$ is freely
homotopic to a single other closed orbit of $\Phi_s$. \ ii) If $\beta'$ intersects
the union of $\{ \alpha_i \}$ then $\beta$ is freely homotopic to infinitely
many other closed orbits of $\Phi_s$.
\end{theorem}

One can also combine the two constructions above.

{\footnotesize
{
\setlength{\baselineskip}{0.01cm}

\noindent
Florida State University

\noindent
Tallahassee, FL 32306-4510

}
}

\end{document}